\documentclass{amsart}
\usepackage{amssymb,amsmath,amsthm}

\theoremstyle{plain}
\newtheorem{theorem}{Theorem}
\newtheorem{corollary}{Corollary}
\newtheorem{proposition}{Proposition}
\newtheorem{lemma}{Lemma}

\theoremstyle{definition}
\newtheorem{definition}{Definition}
\newtheorem{example}{Example}

\newcommand\Z{{\mathbb Z}}

\renewcommand\P{{\mathcal{P}}}
\newcommand\PF{{\mathcal{PF}}}
\newcommand\RP{{\mathcal{RP}}}
\newcommand\PP{{\mathcal{PP}}}
\newcommand\PPF{{\mathcal{PPF}}}
\newcommand\PRP{{\mathcal{PRP}}}

\begin{document}

\title{Polynomial functions on the units of $\Z_{2^n}$}

\author{Smile Markovski}

\address{Ss.~Cyril and Methodius University, Faculty of Sciences,
Institute of Informatics, P.O.~Box 162, 1000 Skopje, Republic of
Macedonia\ \  \&\ \ Erasmus Mundus Scholarship, Institute of
Telematics, NTNU, Trondheim, Norway}

\email{smile@ii.edu.mk}

\author{Danilo Gligoroski}
\address{Centre for Quantifiable Quality of Service in Communication
Systems, Norwegian University of Science and Technology,
O.S.Bragstads plass 2E, N-7491 Trondheim, Norway}

\email{gligoroski@yahoo.com}

\thanks{The visit of the first two
authors to the ``Special semester on Gr\" obner bases - Gr\" obner
Bases in Cryptography, Coding Theory, and Algebraic Combinatorics'',
April 30 - May 06, 2006 in Linz, Austria, organized by RISC and
RICAM, was very helpful and stimulated some of the ideas that are
presented in this paper.}.

\author{Zoran {\v S}uni\'c}

\address{Department of Mathematics, Texas A\&M University, College Station,
TX 77843-3368, USA}

\email{sunic@math.tamu.edu}

\thanks{The third author was partially supported by NSF grant DMS-0600975}


\begin{abstract}
Polynomial functions on the group of units $Q_n$ of the ring
$\Z_{2^n}$ are considered. A finite set of reduced polynomials
$\RP_n$ in $\Z[x]$ that induces the polynomial functions on
$Q_n$ is determined. Each polynomial function on $Q_n$ is induced by
a unique reduced polynomial - the reduction being made using a
suitable ideal in $\Z[x]$. The set of reduced polynomials
forms a multiplicative 2-group. The obtained results are used to
efficiently construct families of exponential cardinality of, so
called, huge $k$-ary quasigroups, which are useful in the design of various types of
cryptographic primitives. Along the way we provide a new (and simpler) proof of a result of Rivest characterizing  the permutational polynomials on $\Z_{2^n}$. 
\end{abstract}

\maketitle

\section{Introduction}

The need for new kinds of computational methods and devices is
growing as a result of the possibility of their application in the
new developing fields in mathematics and computer science, in
particular cryptography and coding theory. Finite fields and integer
quotient rings are traditionally used for such computational needs.
The integer quotient rings are somewhat disadvantaged due to the
fact that their nonzero multiplicative structure does not form a
group (except when they happen to be fields). The structure of the
ring of polynomials over rings, and especially over integer quotient
rings, has been under investigation for almost a century. Let us mention here
chronologically some of the authors: Kempner (1921)~\cite{Kempner},
N\" obauer (1965)~\cite{Nobauer}, Keller and Olson (1968)
\cite{keller}, Mullen and Stevens (1984)~\cite{mullen}, Rivest
(2001)~\cite{rivest}, Bandini (2002)~\cite{bandini}, Zhang (2004)
\cite{Zhang}. We emphasize that the paper of Rivest~\cite{rivest} is closest to our work and his results can be inferred from ours (see Section~\ref{s:rivest}). 

We consider its group of units $Q_n$ in $\Z_{2^n}$ and define a finite set $\RP_n$ of
reduced polynomials over $\Z$ that induce the set $\PF_n$ of all polynomial
functions that keep $Q_n$ invariant. The set $\RP_n$ is a finite
2-group under polynomial multiplication modulo functional
equivalence. Exactly half of the reduced polynomials induce
permutations on $Q_n$.

The reduced polynomials are obtained by using an ideal $I_n$ in $\Z[x]$ such that every polynomial in $I_n$ induces the 0 constant function on $Q_n$ and two polynomials are functionally
equivalent over $Q_n$ if and only if they are equivalent with
respect to the ideal $I_n$. 

By using our reduction algorithms we are able to give efficient answers to several
problems. We show that there are efficient algorithms (polynomial
complexity with respect to the input parameters) for the following
problems:

(i) given a polynomial inducing a polynomial function on $Q_n$,
determine the reduced polynomial inducing the same polynomial
function,

(ii) given a polynomial inducing a permutation on $Q_n$, determine
the reduced polynomial inducing the inverse permutation.

(iii) given a polynomial inducing a polynomial function on $Q_n$,
determine the reduced polynomial for the multiplicative inverse.

In the last part of the paper we use the obtained results to
construct families of quasigroups of large cardinality. 
We define the concept of huge quasigroups as quasigroups of large order that can be handled 
effectively, in the sense that the multiplication in the
quasigroup, as well as in its adjoint operations, can be effectively
realized (polynomial complexity with respect of $\log n$, where $n$
is the order of the quasigroup). The need for permutations and
quasigroups of large (huge) orders such as $2^{16}$, $2^{32}$,
$2^{64}$, $2^{128}$, that can be easily handled is associated with
the development of the modern massively produced 32-bit and 64-bit
processors. Strong links between modern cryptography and quasigroups
(equivalently, Latin squares) have been observed by
Shannon~\cite{shannon} more than 50 years ago. Subsequently, the
cryptographic potential of quasigroups in the design of different
types of cryptographic primitives has been addressed in numerous
works. Authentication schemas have been proposed by D\`enes and
Keedwell (1992)~\cite{Denes1992}, secret sharing schemes by Cooper,
Donovan and Seberry (1994)~\cite{Cooper1994}, a version of popular
DES block cipher by using Latin squares  by Carter, Dawson, and
Nielsen (1995)~\cite{Carter1995}, different proposals for use in the
design of cryptographic hash functions by several
authors~\cite{Schnorr1994,hash2005,hash2006}, a hardware stream
cipher by Gligoroski, Markovski, Kocarev and Gusev (2005)
\cite{stream}.

We want to emphasize that the results in this work concerning
effective constructions of large quasigroups, besides in
cryptography, can also be of interest in other areas (such as coding
theory, design theory, ...).

\subsection{Organization of the content}

Well known background on the structure of the group $Q_n$ and on Hensel lifting (useful to extract inverses in $Q_n$) 
is presented in Section~\ref{s:group}. Full description of the polynomials in $\Z[x]$ that induce transformations on $Q_n$ (and the finite set of reduced polynmials that represent them) is provided in Section~\ref{s:pol}, while the polynomials in $\Z[x]$ that induce permutations on $Q_n$ are characterized in Section~\ref{s:perm}. Section~\ref{s:rivest} is a brief interlude in which we use our results to present a new proof or a result of Rivest~\cite{rivest} providing a characterization of polynomials in $\Z[x]$ that induce permutations on $\Z_{2^n}$. The group of reduced polynomials under multiplication is briefly considered in Section~\ref{s:multiplication}. Section~\ref{s:algorithmic_aspects} provides polynomial algorithms that handle construction of reduced polynomials related to interpolation, functional inversion, and multiplicative inversion. Finally, applications to effective constructions of large $k$-ary quasigroups are provided in Section~\ref{s:quasigroups}. 

\section{The group $(Q_n,\cdot)$}\label{s:group}

The integer quotient ring $(\Z_k,+,\cdot)$, where $k$ is a positive
integer, is a well known mathematical structure, where the addition
and multiplication are interpreted modulo $k$. This ring is
associative and commutative ring with a unit element $1$. Here we
are concerned solely with the case $k=2^n.$ The set
$Q_n=\{1,3,\dots,2^n-1\}$ is a subgroup of the multiplicative
semigroup $(\Z_{2^n},\cdot)$. Indeed, $Q_n$ is precisely the group
of units of $\Z_{2^n}$. Note that if $n=1$, then $Q_n$ is trivial, and if $n=2$, $Q_2 = \Z_2 = \langle -1 \rangle$. The structure of the abelian group $Q_n$, for $n \geq 3$, is
given by the following result.

\begin{proposition}\label{the group}
Let $n \geq 3$. Then $(Q_n,\cdot) \cong \Z_2 \times \Z_{2^{n-2}}$.

Moreover, $Q_n$ is generated by $-1$ and $5$, the order of -1 is 2,
and the order of 5 is $2^{n-2}$.
\end{proposition}
\begin{proof}
The subset $F_n \subseteq Q_n$ of numbers of the form $4k+1$ forms a
subgroup of index 2 in $Q_n$. Since $5 \in F_n$, we have
$5^{2^{n-2}}=1$ in $Q_n$. On the other hand,
\[
 5^{2^{n-3}} = (4+1)^{2^{n-3}} = \sum_{i=0}^{2^{n-3}}
 \binom{2^{n-3}}{i}2^{2i}.
\]
The highest power of 2 dividing $i!$ is $\lfloor i/2 \rfloor +
\lfloor i/4 \rfloor + \cdots < i/2+i/4+\cdots = i$. Thus each of the
terms $\binom{2^{n-3}}{i}2^{2i}$ is divisible by $2^{n-3+2i-(i-1)} =
2^{n-2+i}$ and we have
\begin{equation}
 5^{2^{n-3}} \equiv 1+2^{n-3}\cdot 2^2 \equiv 2^{n-1}+1 \pmod{2^n}.
\end{equation}
Therefore $5^{2^{n-3}} \neq 1$ in $Q_n$, the order of 5 is
$2^{n-2}$, and $F_n$ is a cyclic group generated by 5.

The order of -1 is clearly 2. Since -1 is not in $F_n$ (it has the
form $4k+3$) we have that $Q_n = \langle -1 \rangle \times \langle 5
\rangle = \Z_2 \times \Z_{2^{n-2}}$.
\end{proof}

\begin{corollary}
Let $n \geq 3$. The multiplicative order of every $a \in Q_n$ divides $2^{n-2}$.
\end{corollary}

Given a large value of $n$ and $a\in Q_n$, can we effectively find
the inverse $a^{-1}$? Note that if we express $a$ as $a=(-1)^i \cdot
5^j$, for some $i\in \{0,1\}$, $ j\in \{0,1,\dots,2^{n-2}-1\}$, then
its inverse in $Q_n$ is given by
\[ a^{-1} = (-1)^i \cdot 5^{2^{n-2}-j}.\]
However, this requires representing $a$ in the form $a=(-1)^i
\cdot 5^j$, for some $i\in \{0,1\}$. It is fairly easy to decide
if $i=0$ or $i=1$. Indeed, $i=0$ when $a$ is of the form $4k+1$
and $i=1$ otherwise. However, to determine $j$ we need to solve a
discrete logarithm problem of the type $5^x = a \pmod{2^n}$. This
apparent difficulty can be sidestepped by calculating the inverse
by applying Hensel lifting~\cite{Perron48} (also known as
Newton-Hensel lifting~\cite{Kaltofen85}).

The basic idea is to use binary representation of the integers
modulo $2^n$. Given $r\in \mathbb{Z}_{2^n}$, its binary
representation is $r_{n-1}r_{n-2}\dots r_1r_0$, where
$r_j\in\{0,1\}$ is the $(j+1)-$th bit of $r$. In the same way, the
binary representation of a variable $x$ is given by
$x_{n-1}x_{n-2}\dots x_1x_0$, where $x_j$ are bit variables. Now,
let $r$ be a root of the polynomial $P(x)$. Then $P(x)=(x-r)S(x)$
for some polynomial $S(x)$. The equality $P(x)=(x-r)S(x)$ in the
ring $\mathbb{Z}_{2^k}$, where $k<n$, is given by
\[
 P(x_{k-1}\dots x_1x_0) =
  (x_{k-1}\dots x_1x_0-r_{k-1}\dots r_1r_0)S(x_{k-1}\dots x_1x_0).
\]
The last equality shows that if we want to find the $k$ least
significant bits of a root $r$ of $P(x)$, we need to consider the
equation $P(x)=0$ in the ring $\mathbb{Z}_{2^k}$.

One variant of the Hensel lifting algorithm for finding a root of
$P(x)$ is the following:\\

{\it Step $1$: Determine a bit $r_0$ such that $P(r_0)=0$ in
$\mathbb{Z}_2$.}\\

This can be accomplished simply by checking if $P(0)=0$ or $P(1)=0$ (or both!) in $\Z_2$.\\

Let the bits $r_0,\dots,r_{k-1}$ be already chosen in {\it Step 1 - Step $k$.}\\

{\it Step $k+1$: Determine a bit $r_k$ such that $P(r_k
r_{k-1}\dots r_0)=0$ in $\mathbb{Z}_{2^{k+1}}$.}\\

Since the bits $r_0,\dots,r_{k-1}$ are known, this can be
accomplished by checking if $P(0r_{k-1}\dots r_0)=0$
or $P(1r_{k-1}\dots r_0)=0$ (or both) in $\Z_{2^{k+1}}$.\\

{\it The algorithm stops after Step $n$.}\\

In order to find all roots of a polynomial one has to follow all the
branching points of the algorithm (whenever both 0 and 1 are good
choices one has to follow both choices, and whenever neither 0 nor 1
are good choices one discards that particular branch of the search).

Given $a\in Q$, the root of the polynomial $ax-1$ is the
inverse of $a$. In this case, the above algorithm has polynomial
complexity in $n$, since there is only one root and the above
algorithm will produce the unique correct bit of $a^{-1}$ at each
step (there is no branching).

\section{Polynomial functions on $Q_n$}\label{s:pol}

Every polynomial $P(x)$ from the polynomial ring $\Z[x]$
induces a polynomial function $p: \Z_{2^n} \to \Z_{2^n}$ by the
evaluation map (taken modulo $2^n$). We are interested here in polynomial functions on
$Q_n$, i.e., polynomial functions $p:Q_n \to Q_n$ induced by
polynomials $P(x)$ in $\Z[x]$ such that $p(Q_n)\subseteq Q_n$.
Denote by $\P_n$ the set of polynomials in $\Z[x]$ that induce
polynomial function on $Q_n$ and denote by $\PF_n$ the set of
corresponding polynomial functions on $Q_n$. We implicitly assume
that $n \geq 2$ (as was already mentioned, $Q_1$ is trivial).

We first determine precisely the polynomials over $\Z$ that
induce polynomial functions on $Q_n$, i.e., we determine $\P_n$.

\begin{proposition}\label{p_n}
Let $P(x) = a_0+a_1x+\dots+a_dx^d$ be a polynomial in $\Z[x]$.
Then $P(x)$ is in $\P_n$ (i.e.~$P(x)$ induces a polynomial function
on $Q_n$) if and only if the sum of the coefficients $a_0 + a_1 +
\dots + a_d$ is odd, which, in turn, is equivalent to the condition
that $p(1)$ is odd.
\end{proposition}

\begin{proof}
For every odd number $a$, all the powers $a^i$, $i=0,\dots,d$ are
also odd. Thus the parity of $p(a) = a_0 + a_1a + \dots + a_d a^d$
is equal to the parity of $a_0 + \dots + a_d$.
\end{proof}

The finite set $\PF_n$ of polynomial functions on $Q_n$ is induced
by the infinite set of polynomials in $\P_n$. We will determine a
finite set of polynomials, that induce all polynomial functions in
$\PF_n$. In order to define this set, we need some preliminary
definitions.

For an integer $i$, define $t_i = \lfloor i/2 \rfloor + \lfloor i/4
\rfloor + \lfloor i/8 \rfloor + \dots$, i.e., $t_i$ is the largest
integer $\ell$ such that $2^\ell$ divides $i!$. Let $d_n$ be the largest
integer $i$ such that $n-i-t_i$ is positive.

\begin{definition}
A polynomial $P(x)=a_0+a_1x+\dots+a_dx^d$ in $\P_n$ is called
\emph{reduced} if

(i) the degree of $P(x)$ is no higher than $d_n$.

(ii) $0 \leq a_i \leq 2^{n-i-t_i}-1$, for $i=0,\dots,d_n$.

Denote the set of reduced polynomials in $\P_n$ by $\RP_n$.
\end{definition}

\begin{proposition}
The number of reduced polynomials in $\RP_n$ is
\[
 |\RP_n|= 2^{(2n-d_n)(d_n+1)/2-1-\sum_{i=0}^{d_n} t_i}.
\]
\end{proposition}

\begin{proof}
The number of polynomial of degree at most $d_n$ with restrictions
on the coefficients given by (ii) is
\[
 2^{\sum_{i=0}^{d_n} n-i-t_i} = 2^{n(d_n+1)-d_n(d_n+1)/2-\sum_{i=0}^{d_n} t_i}.
\]
Exactly half of such polynomials also satisfies the condition
required by Proposition~\ref{p_n} on the parity of the sum of the
coefficients. Indeed, we can match up any polynomial $P(x)=a_0+a_1x+ \dots + a_dx^d$ in
that satisfies the conditions (i) and (ii) with the
polynomial $P(x)+1$ if $a_0$ is even and with $P(x)-1$ if $a_0$ is odd. In both cases, the obtained polynomial also satisfies the conditions (i) and (ii). In such
a matching exactly one polynomial in each pair has odd sum of
coefficients.
\end{proof}

Two polynomials $P(x)$ and $T(x)$ in $\P_n$ are said to be
\emph{functionally equivalent} over $Q_n$ if they induce the same
polynomial function on $Q_n$. In that case we write $P(x) \approx
T(x)$. Clearly, $\approx$ is an equivalence relation on $\P_n$.

The polynomials $P(x)$ and $T(x)$ are functionally equivalent over
$Q_n$ if and only if the difference $P(x)-T(x)$ induces the constant
0 function on $Q_n$. With this in mind, we define now a finite set
of polynomials over $\Z$ that induce the 0 constant function
on $Q_n$.

\begin{definition}
For $i = 0,\dots,d_n$, define the polynomial
\[ P_{n,i}(x) = 2^{n-i-t_i}(x+1)(x+3)\dots (x+2i-1) \]
of degree $i$. When $i=0$ the understanding is that $P_{n,0} = 2^n$.
Define also the polynomial
\[ P_{n,d_n+1}(x) = (x+1)(x+3)\dots (x+2d_n+1)\]
of degree $d_n+1$.
\end{definition}

Denote the ideal generated by $P_{n,i}(x)$, $i=0,\dots,d_n+1$, in
$\Z[x]$ by $I_n$. Thus 
\[
 I_n=
 \left\{ \sum_{i=0}^{d_n+1} S_i(x)P_{n,i}(x) \mid S_i(x) \in \Z[x], \ i=0,\dots,d_n+1 \right\}. 
\]

\begin{proposition}\label{back}
Every polynomial in $I_n$ induces the 0 constant function on $Q_n$.
\end{proposition}

\begin{proof}
What we need to prove is that, for every $x \in Q_n$
\[ p_{n,i}(x) \equiv 0 \pmod{2^n}. \]
This is clear since, for any  $x \in Q_n$ the product $(x+1)(x+3)
\dots (x+2i-1)$ is a product of $i$ consecutive even numbers and it
is therefore divisible by $2^i i!$, implying that it is divisible by
$2^{i+t_i}$. For $i=0,\dots,d_n$ we then have that $p_{n,i}(x)$ is
divisible by $2^{n-i-t_i}\cdot 2^{i+t_i} = 2^n$. For $i=d_n+1$, we
have that $n \leq i+t_i$, and therefore $2^n$ divides $p_{n,i}(x)$
in this case as well.
\end{proof}

We state now the two main results of this section. 

\begin{theorem}\label{functional equivalence}
Two polynomials $P(x)$ and $T(x)$ in $\P_n$ are functionally
equivalent over $Q_n$ if and only if $P(x)-T(x)$ is a member of
$I_n$.
\end{theorem}

\begin{theorem}\label{t:unique}
Every polynomial function in $\PF_n$ is induced by a unique reduced
polynomial in $\RP_n$. 
\end{theorem}

We will prove the Theorem~\ref{functional
equivalence} and Theorem~\ref{t:unique} through a series of lemmas and propositions. Along the way we provide some additional information (for instance Proposition~\ref{d_n bound} establishes a linear upper bound on the degree of a reduced polynomial). While some other approaches are certainly possible, we chose to follow a simple constructive route, since we are interested in algorithmic/complexity issues (see Section~\ref{s:algorithmic_aspects}). 

\begin{proof}[Proof of Theorem~\ref{functional equivalence}, sufficiency]
If $P(x)-T(x)$ is in $I_n$ then, by Proposition~\ref{back},
$P(x)-T(x)$ induces the constant 0 function on $Q_n$, implying that $P(x)$ and $Q(x)$ are functionally equivalent over $Q_n$. 
\end{proof}

\begin{proposition} \label{reduced}
Every polynomial function in $\PF_n$ is induced by a reduced
polynomial in $\RP_n$.

Moreover, for every polynomial $P(x)$ in $\Z[x]$ there exists a polynomial $S_P(x)$ in $I_n$ such that $P(x) - S_P(x)$ is reduced and functionally equivalent to $P(x)$ over $Q_n$. 
\end{proposition}

\begin{proof}
Let $p(x)$ be a polynomial function in $\PF_n$ induced by the
polynomial $P(x)$.

If the degree $d$ of $P(x)$ is higher than $d_n$ we may replace
$P(x)$ by $P(x) - a_d x^{d-d_n-1} P_{n,d_n+1}$, where $a_d$ is the
coefficient of $x^d$ in $P(x)$. The polynomial $P(x) - a_d
x^{d-d_n-1} P_{n,d_n+1}$ has degree smaller than $d$ and is
functionally equivalent to $P(x)$. We may continue this until we
obtain a polynomial that is functionally equivalent to $P(x)$ and
has degree no higher than $d_n$.

We assume now that $P(x)$ has degree no higher than $d_n$. If
$P(x)$ is reduced we are done. Otherwise, let $i$ be the highest
degree of a coefficient $a_i$ of $x^i$ that does not satisfy the
requirement $0 \leq a_i \leq 2^{n-i-t_i} -1$. If $q$ is the
quotient obtained by dividing $a_i$ by $2^{n-i-t_i}$ then $P(x)
\approx P(x) - qP_{n,i}$, and the coefficient at degree $i$ in
$P(x) - qP_{n,i}$ is in the correct range $0,\dots,2^{n-i-t_i}-1$.

We repeat this procedure with the next highest degree that has a
coefficient out of range until we reach a reduced polynomial that is
functionally equivalent to $P(x)$.
\end{proof}

\begin{example}\label{rewriting}
Let $n=5$. We have $0+t_0=0$, $1+t_1=1$, $2+t_2=3$, $3+t_3=4$ and
$4+t_4=7$. Therefore $d_5=3$, and every reduced polynomial has the
form
\[ R(x) = a_0 + a_1x + a_2x^2 +a_3x^3, \]
where $0\leq a_0 \leq 31$, $0 \leq a_1 \leq 15$, $0 \leq a_2 \leq 3$
and $0 \leq a_3 \leq 1$. The polynomials $P_{5,i}(x)$, $i=0,1,2,3,4$
are given by
\begin{align*}
 P_{5,0}(x) &= 2^5 = 32, \\
 P_{5,1}(x) &= 2^4(x+1) = 16+16x, \\
 P_{5,2}(x) &= 2^2(x+1)(x+3) = 12+16x+4x^2, \\
 P_{5,3}(x) &= 2(x+1)(x+3)(x+5) = 30+14x+18x^2+2x^3\\
 P_{5,4}(x) &= (x+1)(x+3)(x+5)(x+7) = 9+16x+22x^2+16x^3+x^4.
\end{align*}

Then, for the polynomial $P(x) = 3x^5 +1$, we have
\begin{align*}
 P(x) &= 1+3x^5 \approx (1+3x^5) - 3xP_{5,4}(x) \\
      &\approx 1+5x+16x^2+30x^3+16x^4 \approx (1+5x+16x^2+30x^3+16x^4) - 16P_{5,4}(x) \\
      &\approx 17+5x+16x^2 +30x^3 \approx (17+5x+16x^2 +30x^3) - 15P_{5,3}(x) \\
      &\approx 15+19x+2x^2 \approx (15+19x+2x^2) - P_{5,1}(x) \\
      &\approx 31+3x+2x^2.
\end{align*}

The calculations are done modulo 32 all the time. This is equivalent
to using $P_{5,0}=32$ to make reductions.
\end{example}

\begin{proposition}\label{d_n bound}
Every polynomial function in $\PF_n$ is induced by a polynomial of
degree smaller than $(n+1+\lfloor \log_2 n \rfloor)/2$.
\end{proposition}

\begin{proof}
We need to prove that $d_n < (n+1+\lfloor \log_2 n \rfloor)/2$.

First note that $i-1 - \lfloor \log_2 i \rfloor \leq t_i$. Indeed
$t_i = \lfloor i/2 \rfloor + \lfloor i/4 \rfloor + \dots$~. Only the
first $\lfloor \log_2 i \rfloor$ terms of the series are possibly
positive. Thus $t_i = \sum_{k=1}^{\lfloor \log_2 i \rfloor} \lfloor
i/2^k \rfloor > \sum_{k=1}^{\lfloor \log_2 i \rfloor}  (i/2^k -1) =
i \left(1-\frac{1}{2^{\lfloor \log_2 i \rfloor}}\right) - \lfloor
\log_2 i \rfloor > i \left(1-\frac{1}{2^{\log_2 i -1}}\right) -
\lfloor \log_2 i \rfloor = i-2 - \lfloor \log_2 i \rfloor$.

Assume that $n \geq i \geq \frac{n+1+ \lfloor \log_2 n
\rfloor}{2}$. Then
\[
 i+t_i \geq 2i -1 - \lfloor \log_2 i \rfloor \ge 2\frac{n+1+ \lfloor \log_2 n
 \rfloor}{2} -1 - \lfloor \log_2 n \rfloor = n.
\]
Since $d_n$ is the largest
integer $i$ such that $n-i-t_i$ is positive, we must have $d_n < \frac{n+1+ \lfloor \log_2 n \rfloor}{2}$.
\end{proof}

\begin{lemma}\label{row reduction}
Let $M_m$ be the $(m+1) \times (m+1)$ Vandermonde matrix
\[
 M_m =  \begin{bmatrix}
         1 & 1 & 1 & \dots   & 1 \\
         1 & 3 & 3^2 & \dots & 3^m \\
         \vdots & \vdots & \vdots & \ddots & \vdots\\
         1 & (2m+1) & (2m+1)^2 & \dots & (2m+1)^m
        \end{bmatrix},
\]
in which the rows and columns are indexed by $0,\dots,m$. The matrix
$M_m$ is row equivalent over $\Z$ to a matrix of the form
\[
 R_m =  \begin{bmatrix}
         1 & * &  \dots & * \\
         0 & 2 &  \dots & * \\
         \vdots & \vdots & \ddots & \vdots\\
         0 & 0 & \dots & 2^m m!
        \end{bmatrix},
\]
where the $*$'s represent integers (whose values are irrelevant for
our purposes), and the only type of row reduction used is the one in
which an integer multiple of a row is added to another row.
\end{lemma}

\begin{proof}
We will prove, by induction on $m$, that

(i) every vector $r_{i,m}=(1,2i+1,\dots,(2i+1)^m)$, $i \geq m+1$, is
a linear combination of the rows $0,\dots,m$ in $M_m$,

(ii) the matrix $R_m$ can be obtained by row reduction of the
indicated type from $M_m$.

(iii) assuming $r_{i,m} = \alpha_0 r_{0,m} + \dots + \alpha_{m}
r_{m,m}$ in (i),
\[
 r_{i,m+1} - (\alpha_0 r_{0,m+1} + \dots + \alpha_{m}r_{m,m+1})
 = (0,0,\dots,0,s_i),
\]
where $s_{m+1}=2^{m+1}(m+1)!$ and $s_i$ is divisible by
$2^{m+1}(m+1)!$ if $i \geq m+2$.

The claims (i),(ii),(iii) are clear for $m=0$ and assume they are
valid for some $m \geq 0$. We proceed to the inductive step.

(i) Consider the vector $r_{i,m+1} = (1,2i+1,\dots,(2i+1)^{m+1})$,
$i \geq m+2$. From the inductive assumption (iii),
\[
 r_{i,m+1} - (\alpha_0 r_{0,m+1} + \dots + \alpha_{m}r_{m,m+1})
 = (0,0,\dots,0,s_i)
\]
and
\[ r_{m+1,m+1} - (\alpha_0' r_{0,m+1} + \dots + \alpha_{m}'r_{m,m+1})
 = (0,0,\dots,0,2^{m+1}(m+1)!).
\]
Since $2^{m+1}(m+1)!$ divides $s_i$ we see that $r_{i,m+1}$ can be
indeed written as a linear combination of the rows $0,\dots,m+1$
in $M_{m+1}$.

(ii) Since, from inductive assumption (iii),
\[ r_{m+1,m+1} - (\alpha_0' r_{0,m+1} + \dots + \alpha_{m,m}'r_{m,m+1})
 = (0,0,\dots,0,2^{m+1}(m+1)!).
\]
we see that $M_{m+1}$ is row equivalent to a matrix $R_{m+1}'$ in
which the bottom row is $(0,0,\dots,0,2^{m+1}(m+1)!)$ and the
upper left block of size $(m+1) \times (m+1)$ is $M_m$. The
inductive assumption (ii) shows that $R_{m+1}'$ is row equivalent
to $R_{m+1}$.

(iii) Consider the matrix $M_{m+2}(i)$ obtained from $M_{m+1}$ by
extending it by the column vector $(1,3^{m+2},\dots,(2m+3)^{m+2})$
on the right and then by the row vector $r_{i,m+2}$, $i \geq m+2$,
at the bottom. The new matrix is the $(m+3)\times (m+3)$
Vandermonde matrix corresponding to the values $1,3,5,\dots,2m+3$
and $2i+1$. From parts (i) and (ii) of the inductive step that we
just proved, we know that $M_{m+2}(i)$ is row equivalent to a
matrix $R_{m+2}(i)$ in which the bottom row is $(0,0,\dots,s_i)$,
for some integer $s_i$, and the upper left block of size $(m+2)
\times (m+2)$ is $R_{m+1}$. The determinant of the Vandermonde
matrix $M_{m+2}(i)$ is equal to
\begin{alignat*}{2}
 \det(M_{m+2}(i)) &= &&(3-1)\cdot(5-3)(5-1) \dots ((2m+3)-(2m+1))\dots((2m+3)-1)\\
                  &  &&((2i+1)-(2m+3))\dots ((2i+1)-1) \\
                  &= &&\det(M_{m+1}) \cdot ((2i+1)-(2m+3))\dots ((2i+1)-1).
\end{alignat*}
On the other hand, the row equivalence of $M_{m+2}(i)$ and
$R_{m+2}(i)$ shows that
\[
 \det(M_{m+2}(i)) = \det(R_{m+2}(i)) =
 \det(R_{m+1})\cdot s_i = \det(M_{m+1}) \cdot s_i.
\]
Since $\det(M_{m+1}) \neq 0$ we obtain that
$s_i=((2i+1)-(2m+3))\dots ((2i+1)-1)$. In case $i=m+2$,
$s_{m+2}=2\cdot 4 \cdot \dots \cdot (2(m+2)) = 2^{m+2}(m+2)!$. If
$i \geq m+3$, then $s_i$ is a product of $m+2$ consecutive even
numbers and is therefore divisible by $2^{m+2}(m+2)!$. The
inductive claim (iii) now easily follows.
\end{proof}

\begin{proof}[Proof of Theorem~\ref{t:unique}, uniqueness]
Let $p$ be a polynomial function in $\PF_n$. All reduced polynomials
inducing $p$ are given by
\[ P(x) = a_0 + a_1x + \cdots +a_dx^{d}, \]
where $d=d_n$, and the coefficients $a_0,\dots,a_d$ satisfy the
linear system
\[ M_d (a_0,a_1,\dots,a_d)^T = (p(1),p(3),\dots,p(2d+1))^T,  \]
where $(.)^T$ stands for transposition. By Lemma~\ref{row
reduction}, this system is equivalent in $\Z_{2^n}$ to the upper
triangular system
\[ R_d (a_0,a_1,\dots,a_d)^T = (b_0,b_1,\dots,b_d)^T, \]
where $b_i$ are some elements in $\Z_{2^n}$. Since odd numbers are
units in $\Z_{2^n}$ this system is equivalent to a triangular system
\[ R'_d (a_0,a_1,\dots,a_d)^T = (b'_0,b'_1,\dots,b'_d), \]
where
\begin{equation}\label{triangular}
 R'_d =  \begin{bmatrix}
         2^{0+t_0} & * &  \dots & * \\
         0 & 2^{1+t_1} &  \dots & * \\
         \vdots & \vdots & \ddots & \vdots\\
         0 & 0 & \dots & 2^{d+t_d}
        \end{bmatrix}.
\end{equation}

The last equation of this system now reads $2^{d+t_d}a_d = b'_d$.
Since $0 \leq a_d \leq 2^{n-d-t_d}-1$ this equation can only have
one solution in $\Z_{2^n}$. We can substitute this solution in the
second to last equation to obtain an equation
$2^{d-1+t_{d-1}}a_{d-1} = b''_{d-1}$, which will also have a unique
solution in $\Z_{2^n}$ since $0 \leq a_{d-1} \leq
2^{n-d-1-t_{d-1}}-1$.

Continuing with the backward substitution in the triangular system
with matrix $R'_d$ we obtain a unique solution for all the
coefficients $a_d,a_{d-1},\dots,a_0$ of $P(x)$.
\end{proof}

\begin{proposition}
The number of polynomial functions in $\PF_n$ is equal to the number
of reduced polynomials in $\RP_n$.
\end{proposition}

\begin{example}\label{135}
Let $n=4$. In this case $d=d_4=2$. Let $p$ be a polynomial function
in $\PF_4$ for which $p(1)=9$, $p(3)=5$ and $p(5)=9$. We are trying
to determine the unique reduced polynomial $P(x)=a_0+a_1x+a_2x^2$ in
$\RP_4$ that induces $p$. Note that the coefficients must satisfy
the range conditions $0 \leq a_0 \leq 15$, $0 \leq a_1 \leq 7$, and
$0 \leq a_2 \leq 1$. The known values of $p$ give the system
\[
 \begin{bmatrix}
 1 & 1 & 1 &| & 9 \\
 1 & 3 & 9 &| & 5 \\
 1 & 5 & 9 &| & 9
 \end{bmatrix},
\]
which is row equivalent to
\[
 \begin{bmatrix}
 1 & 1 & 1 &| & 9 \\
 0 & 2 & 8 &| & 12 \\
 0 & 0 & 8 &| & 8
 \end{bmatrix}.
\]
The last equation $8a_2=8$, together with the condition $0 \leq a_2
\leq 1$, gives $a_2=1$. The second equation $2a_1+8a_2=12$, together
with the conditions $a_2=1$ and $0 \leq a_1 \leq 7$, gives $a_1=2$.
Finally, the first equation $a_0+a_1+a_2=9$, together with the
conditions $a_2=1$, $a_1=2$ and $0 \leq a_0 \leq 15$, gives $a_0=6$.
Thus the unique reduced polynomial inducing $p$ is $P(x)= 6 + 2x +
x^2$.
\end{example}

\begin{example}
It is clear that one can uniquely determine the reduced polynomial
$R(x)$ that is functionally equivalent to $P(x)$ from the value of
$p$ at any $d_n+1$ consecutive values of $x$.

On the other hand, not any $d_n+1$ values are sufficient. Indeed,
let $n=4$ and $p$ be a polynomial function in $\PF_4$ for which
$p(1)=9$, $p(5)=9$ and $p(9)=9$. We are trying to determine a
reduced polynomial $R(x)=a_0+a_1x+a_2x^2$ in $\RP_4$ that induces
$p$. The known values of $p$ give the system
\[
 \begin{bmatrix}
 1 & 1 & 1 &| & 9 \\
 1 & 5 & 9 &| & 9 \\
 1 & 9 & 1 &| & 9
 \end{bmatrix},
\]
which, together with the range conditions $0 \leq a_0 \leq 15$, $0
\leq a_1 \leq 7$, and $0 \leq a_2 \leq 1$, gives the following 4
solutions: $R(x)=9$, $R(x)=6+2x+x^2$, $R(x) = 5+4x$, $R(x)= 2 + 6x +
x^2$. Note than one of these is the solution obtained in
Example~\ref{135}.
\end{example}

\begin{proof}[Proof of Theorem~\ref{functional equivalence}, necessity]
Let $P(x)$ and $T(x)$ be two functionally equivalent polynomials.
By Proposition~\ref{reduced}, there exists
polynomials $S_P(x)$ and $S_T(x)$ in $I_n$ such that $P(x) - S_P(x)$
and $T(x) - S_T(x)$  are reduced polynomials which are functionally
equivalent to $P(x)$ and $T(x)$. Theorem~\ref{t:unique}
then shows that $P(x)-S_P(x) = T(x) - S_T(x)$, implying that $P(x) -
T(x) = S_P(x)- S_T(x) \in I_n$.
\end{proof}

\begin{proposition}
The set of polynomials in $\Z_{2^n}[x]$ that induce the 0 constant
function on $Q_n$ is precisely the ideal $I_n$.
\end{proposition}
\begin{proof}
We already know from Proposition~\ref{back} that the polynomials in
$I_n$ induce the constant 0 function on $Q_n$. Conversely, let $P(x)$
induce the constant 0 function on $Q_n$. By Proposition~\ref{reduced} there exists a polynomial $S_P(x)$ in $I_n$ such that $P(x)-S_P(x)$ is reduced and functionally equivalent to $S(x)$. Since the zero polynomial is reduced,  we must have $P(x) - S_P(x) = 0$, by the uniqueness property in Theorem~\ref{t:unique}. Therefore $P(x) = S_P(x) \in I_n$. 
\end{proof}

\section{Permutational polynomial functions on $Q_n$}\label{s:perm}

Some polynomial function on $Q_n$ are permutations on $Q_n$. Denote
the set of such (permutational) polynomial functions by $\PPF_n$ and
the set of polynomials over $\Z$ inducing such functions by $\PP_n$.

\begin{proposition}\label{pp_n}
Let $P(x) = a_0+a_1x+\dots+a_dx^d$ be a polynomial in $\P_n$. Then
$P(x)$ is in $\PP_n$ (i.e.~$P(x)$ induces a permutational polynomial
function on $Q_n$) if and only if the sum of the odd indexed
coefficients $a_1+a_3+a_5+\cdots$ is an odd number.
\end{proposition}

\begin{proof}
Let $a,b \in Q_n$. We have
\begin{align*}
 p(a)-p(b) &= a_1(a-b)+a_2(a^2-b^2) + \dots + a_d(a^d - b^d) = \\
 &= (a-b)(a_1A_1+a_2A_2+\dots+a_dA_d),
\end{align*}
where $A_1=1$ and $A_i=a^{i-1}+a^{i-2}b+\cdots+ab^{i-2}+b^{i-1}$,
for $i \geq 2$. The number $A_i$ is even if and only if $i$ is even.
Consequently, $a_1A_1+a_2A_2+\dots+a_dA_d$ is odd if and only if
$a_1+a_3+a_5+\cdots$ is odd number.

If $a_1+a_3+a_5+\cdots$ is even then
$(a-b)(a_1A_1+a_2A_2+\dots+a_dA_d) \equiv 0 \pmod{2^n}$, for
$a=2^{n-1}+1$, $b=1$. Thus, for this choice of $a$ and $b$, we have 
$p(a)=p(b)$ and, therefore, $p$ is not a permutation on $Q_n$.

If $a_1+a_3+a_5+\cdots$ is odd then
$(a-b)(a_1A_1+a_2A_2+\dots+a_dA_d) \equiv 0 \pmod{2^n}$ if and only
if $a-b \equiv 0 \pmod{2^n}$, i.e., $a=b$ in $Q_n$. Thus $p$ is a
permutation in this case.
\end{proof}

Since we have a bijective correspondence between reduced polynomials
and polynomial functions, it is clear that we also have a bijective
correspondence between the reduced polynomials in $\RP_n$ with odd
sum of odd indexed coefficients and the permutational polynomial
functions in $\PPF_n$.

\begin{proposition}
The number of permutational polynomial functions in $\PPF_n$ is
equal to 
\[  
 |\PPF_n| =2^{(2n-d_n)(d_n+1)/2- 2 - \sum_{i=0}^{d_n}t_i}
\]
\end{proposition}

\begin{example}
Reduced polynomials in $\RP_n$ of degree at most 3 that induce
permutational polynomial functions in $\PPF_n$ have the form
$a_0+a_1x+a_2x^2+a_3x^3$, where $a_1+a_3$ is odd, $a_0+a_2$ is even,
$0 \leq a_0 \leq 2^n-1$, $0 \leq a_1 \leq 2^{n-1}-1$, $0 \leq a_2
\leq 2^{n-3}-1$, and $0 \leq a_3 \leq 2^{n-4}-1$.
\end{example}

\begin{proposition}
The inverse of a permutational polynomial function $p\in \PPF_n$ is
also a polynomial function.
\end{proposition}

\begin{proof}
If $p\in \PF_n$ is a permutation on $Q_n$, then $p\in \sigma(Q_n)$,
where $\sigma(Q_n)$ denotes the full permutation group of $Q_n$. Let
$r$ be the order of $p$ in $\sigma(Q_n)$. Then $p^{-1}=p^{r-1}$ and
therefore, if $p$ is induced by the polynomial $P(x)$, then $p^{-1}$
is induced by the polynomial $\underbrace{P(P(\dots P(}_{r-1}x)))$.
\end{proof}

\begin{example}
A linear permutational polynomial function $p$ has a linear
permutational polynomial function as its inverse. Indeed, if $p$ is
induced by $b+ax$, then $a$ must be odd, $a^{-1}$ exists in $\Z_{2^n}$ and
$p^{-1}$ is induced by the polynomial $-a^{-1}b + a^{-1}x$.
\end{example}

We can use the permutational polynomial functions on $Q_n$ to define
permutations on $\Z_{2^n}$ (this will be useful in our last section). Denote by $Q'_n$ the set $\Z_{2^n}
\setminus Q_n$ (consisting of 0 and all zero divisors in
$\Z_{2^n}$). We can easily conjugate the action of a polynomial
function on $Q_n$ to an action on $Q'_n$. Namely, given a polynomial
function $h:Q_n \to Q_n$, define $h':Q'_n \to Q'_n$ by $h'(x) =
h(x+1)-1$.

Given a permutation $p \in \PF_n$, we can define a permutation
$\hat{p}$ on $\Z_{2^n}$ by
 \begin{equation}\label{p kapa}
  \hat{p}(x) =
    \begin{cases}
      p(x), & x \in Q_n \\
      p'(x), & x \in Q'_n
    \end{cases}.
\end{equation}
More generally, given permutations $p,h \in \PF_n$, a permutation
$f_{p,h}$ on $\Z_{2^n}$ can be defined by

\begin{equation}\label{p h kapa}
 f_{p,h} =
    \begin{cases}
      p(x), & x \in Q_n \\
      h'(x), & x \in Q'_n
    \end{cases}.
\end{equation}

\section{On a result of Rivest}\label{s:rivest}

The main result of Rivest in~\cite{rivest} provides a criterion for a polynomial over
$\Z$ to induce a permutation on $\Z_{2^n}$. We infer now this result from our results. Note that our proof only relies on Proposition~\ref{p_n} and Proposition~\ref{pp_n}, both of which have short and rather elementary proofs. 

\begin{theorem}[Rivest~\cite{rivest}]\label{t:rivest}
A polynomial $P(x)=a_0+a_1x+ \dots +a_dx^d$ of degree $d \geq 1$
over $\Z$ induces a permutation on $\Z_{2^n}$ if and only if
the following conditions are satisfied:

(a) the sum $a_2+a_4+a_6+\dots$ is even

(b) the sum $a_3+a_5+a_7+\dots$ is even

(c) $a_1$ is odd
\end{theorem}

\begin{proof}
If $P(x)$ is a polynomial that permutes $\Z_{2^n}$ then all elements in $Q_n'=\Z_{2^n} \setminus Q_n$ are mapped to elements of $Q_n'$ or all of them are mapped to elements in $Q_n$ depending on the parity of
$a_0$. Let us first characterize those polynomials over $\Z$
that permute both $Q_n$ and $Q_n'$. They are precisely the
polynomials for which

(i) $a_0$ is even

(ii) the sum of all coefficients $a_0+a_1+\dots+a_d$ is odd

(iii) the sum of the odd index coefficients $a_1+a_3+\dots$ is odd

(iv) the sum of the odd index coefficients in $P(x+1)-1$ is odd.

The first condition ensures that $Q_n'$ is invariant, the second
that $Q_n$ is invariant (Proposition~\ref{p_n}), the third that
$P(x)$ induces a permutation on $Q_n$ (Proposition~\ref{pp_n}) and
the last that $P(x)$ induces a permutation on $Q_n'$ (by conjugating
the action from $Q_n'$ to $Q_n$ we can again use
Proposition~\ref{pp_n}). Let $S(x)=P(x+1)-1$. The sum of odd index
coefficients of $S(x)$ is odd exactly when $(S(1)-S(-1))/2$ is
odd. But $(S(1)-S(-1))/2=(P(2)-P(0))/2 =
a_1+2a_2+2^2a_3+\dots+2^{d-1}a_d$, and therefore this condition is
equivalent to $a_1$ being odd. Therefore the conditions (i)-(iv) are
equivalent to

(i') $a_0$ is even

(ii') the sum $a_2+a_4+a_6+\dots$ is even

(iii') the sum $a_3+a_5+a_7+\dots$ is even

(iv') $a_1$ is odd.

Thus, in order to characterize all polynomials that induce a
permutation on $\Z_{2^n}$ we just need to drop the condition that
$a_0$ is even (which allows $Q_n$ and $Q_n'$ to be mapped to each
other, when $a_0$ is odd).
\end{proof}

In fact, we may establish a precise connection between the (permutational) polynomial functions on $Q_n$ and those on $\Z_{2^n}$. 

\begin{proposition}
Let $n \geq 2$. For every pair of polynomials functions $p,h \in \PF_n$, there exists a polynomial function $g$ on $\Z_{2^n}$, such that 
\[ g(x) = f_{p,h}(x), \]
for $x$ in $\Z_{2^n}$. 
\end{proposition}

\begin{proof}
Consider the polynomial 
\[ V_0(x) = 
 \begin{cases} 
 x^{2^{n-2}}, & n \geq 4 \\
 x^4, & n=3,\\
 x^2, & n=2. 
 \end{cases}
\]
We claim that, for the associated polynomial function $v_0(x)$ on $\Z_{2^n}$,   
\[ 
 v_0(x) = 
 \begin{cases}
  1, & x \in Q_n, \\
  0, & x \in Q_n'.
 \end{cases} 
\] 
The claim can be easily verified directly for $n=2,3$. Assume $n \geq 4$. From Proposition~\ref{the group}, it follows that $v_0(x)=1$, for $x \in Q_n$. On the other hand, $2^{n-2} \geq n$, for $n \geq 4$, which then implies that $v_0(x)=x^{2^{n-2}}=0$, for $x \in Q_n'$. 

Let $V_1(x) = 1 - V_0(x)$. For the associated polynomial function $v_1(x)$ we clearly have 
\[ 
 v_1(x) = 
 \begin{cases}
  0, & x \in Q_n, \\
  1, & x \in Q_n'.
 \end{cases} 
\] 
Therefore, if $P(x)$ and $H(x)$ are polynomial representing the polynomial functions $p(x)$ and $h(x)$ then the polynomial 
\[ G(x) = P(x)V_1(x) + H'(x)V_0(x), \]
where $'H(x)=H(x+1)-1$, induces the function $f_{p,h}$, showing that this function is a polynomial function on $\Z_{2^n}$. 
\end{proof}

\begin{corollary}\label{cor z}
Let $n \geq 2$. The number of permutational polynomial functions on $\Z_{2^n}$ is 
\begin{equation}\label{n of permutations on Z} 
2^{(2n-d_n)(d_n+1)- 3 - 2\sum_{i=0}^{d_n}t_i},
\end{equation}
where $t_i$ is the largest integer $\ell$ such that $2^\ell$ divides $i!$, and $d_n$ is the largest integer $i$ such that $n-i-t_i$ is positive.
\end{corollary}

\begin{proof}
Note that the correspondence that associates to each pair of permutational polynomial functions $(p,h)$ on $Q_n$ the element $f_{p,h}$ in the set of permutational polynomial functions on $\Z_{2^n}$ that keep both $Q_n$ and $Q_n'$ invariant is a bijection. Thus, the number of such permutational polynomial functions on $\Z_{2^n}$ is $|\PPF_n|^2$. The number of permutational polynomial functions on $\Z_{2^n}$ is twice larger than this number since we need to take into account the polynomial functions that permute $Q_n$ and $Q_n'$. Thus, the total number is
\[
 2 |\PPF_n|^2 =  2^{(2n-d_n)(d_n+1)- 3 - 2\sum_{i=0}^{d_n}t_i}. \qedhere
\]
\end{proof}

It is interesting to compare the last corollary to earlier results counting permutational polynomial functions on $\Z_{2^n}$. For instance, the following formula is proved in~\cite{keller}. For $n \geq 2$, the number of permutational polynomial functions on $\Z_{2^n}$ is equal to 
\begin{equation}\label{keller}
 2^{3+ \sum_{j=3}^n \beta_j}, 
\end{equation}
where $\beta_j$ is the smallest integer $s$ such that $2^j$ divides $s!$. Combining this with our result yields the identity 
\[ 2 \sum_{i=0}^{d_n}t_i + \sum_{j=3}^n \beta_j =  (2n-d_n)(d_n+1)- 6, \]
for $n \geq 2$. We note that the number of permutational permutations given by the our formula~\eqref{n of permutations on Z} in Corollary~\ref{cor z} seems easier to evaluate than by using~\eqref{keller}, since the summation goes to a smaller bound ($d_n$ rather than $n$) and the summands are easier to compute. 

\section{Multiplication operation on reduced polynomials}\label{s:multiplication}

Here we consider the multiplication operation on the set $\RP_n$ of
reduced polynomials.

We recall that $\RP_n$ is the set of representatives of the
congruences classes of $\P_n$ modulo the functional equivalence
relation $\approx$. In that sense, given $P(x),S(x)\in \RP_n$, we
denote by $P(x) \cdot S(x)$ the corresponding reduced polynomial
inducing the same polynomial function as the product $P(x)S(x)$ of
the polynomials $P(x)$ and $S(x)$. The set $\P_n$ forms a monoid
under polynomial multiplication. Indeed, if the sum of the
coefficient of both $P(x)$ and $S(x)$ is odd, then $p(1)$ and $s(1)$
are odd and therefore so is p(1)s(1), implying that the sum of the
coefficients of $P(x)S(x)$ is also odd. 

\begin{theorem} 
The equivalence $\approx$ is a congruence on $\mathcal{P}_n$. The
factor $(\RP_n,\cdot) = \mathcal{P}_n/\approx$ is a finite
2-group.
\end{theorem}

\begin{proof}
Let $P_i(x) \approx S_i(x)$, for $i=1,2$, $T_P(x) = P_1(x)P_2(x)$,
and $T_S(x) = S_1(x)S_2(x)$. Then $t_P(x) =
p_1(x)p_2(x)=s_1(x)s_2(x)= t_S(x)$. Thus $P_1(x)P_2(x) \approx
S_1(x)S_2(x)$ and $\approx$ is a congruence on $\mathcal{P}$.

For every $a \in Q_n$, we have $a^{2^{n-2}}=1$ in $Q_n$.
Therefore, for any polynomial $P(x)$ in $\mathcal{P}_n$, the
polynomial $P(x)^{2^{n-2}}$ is functionally equivalent to 1. Thus
each reduced polynomial has a multiplicative inverse.
\end{proof}

In order to avoid confusion we denote inverses of polynomial
functions under composition by $(.)^{-1}$, and the inverse of a
reduced polynomial $P(x)$ under multiplication by $\frac{1}{P(x)}$.

The subset $\PRP_n$ of $\RP_n$ consisting of reduced polynomials
that induce permutations on $Q_n$ is not closed under
multiplication. Indeed, $P(x)=2+x$ induces a permutation on $Q_n$,
while $P(x)^2 = 4+4x+x^2$ does not.

\begin{proposition}
The set of reduced permutational polynomials $\PRP_n$ is closed
under multiplicative inversion, i.e., $P(x)\in \PRP_n$ implies
$\frac{1}{P(x)} \in \PRP_n$.
\end{proposition}

\begin{proof}
This directly follows from the fact that different elements in $Q_n$
have different multiplicative inverses.
\end{proof}

\begin{example}
We have $\frac{1}{2+x} = 2+x$ in $\RP_3$, $\frac{1}{4+3x}=
3+3x+x^2$ in $\RP_4$, and $\frac{1}{31+2x+2x^2+x^3+x^4} =
4+7x+2x^2$ in $\RP_5$.
\end{example}

We note that finding the inverse polynomial by using the equality
$\frac{1}{P(x)} = P(x)^{2^{n-2}-1}$ is not effective. We provide an
effective method in the next section.

\section{Algorithmic aspects}\label{s:algorithmic_aspects}

We briefly address the complexity issues related to interpolation of
polynomial functions, inversion of permutational polynomial
functions and multiplicative inversion of polynomials.

\begin{theorem}\label{interpolation}
There exists an algorithm of polynomial complexity in $n$ that,
given the values $p(1),p(3),\dots,p(2d_n+1)$ of a polynomial
function $p$ in $\PF_n$, produces the unique reduced polynomial
$R(x)$ that induces $p$.
\end{theorem}

\begin{proof}
Note that $d_n$ has a linear upper bound in $n$ by
Proposition~\ref{d_n bound}. Running the row reduction on the
$(d_n+1) \times (d_n+1)$ linear system as suggested in the uniqueness part of the proof
of Theorem~\ref{t:unique} takes polynomially many steps
in terms of $n$.
\end{proof}

\begin{theorem}
There exists an algorithm of polynomial complexity in $n+m$ that,
given a polynomial $P(x) \in \P_n$ of degree $m$ (with coefficients reduced modulo $2^n$, i.e., coefficients in the range between 0 and $2^n-1$ inclusive), produces the
unique reduced polynomial $R(x)$ that is functionally equivalent to
$P(x)$.
\end{theorem}

\begin{proof}
By Theorem~\ref{interpolation} it is sufficient to calculate
$p(1),p(3),\dots,p(2d_n+1)$ in polynomially many steps in terms of
$n+m$. This is possible since the degree of $P(x)$ is $m$ and the
calculations are done modulo $2^n$.

Another approach would be to use the reduction algorithm 
suggested in the proof of Proposition~\ref{reduced} and implemented
in Example~\ref{rewriting}.
\end{proof}

\begin{theorem}
There exists an algorithm of polynomial complexity in $n+m$ that,
given a polynomial $P(x)$ in $\PP_n$ of degree $m$ (with coefficients reduced modulo $2^n$), produces the
unique reduced polynomial inducing the inverse polynomial function
$p^{-1}$.\label{inverse perm}
\end{theorem}

\begin{proof}
First calculate $p(1),p(3),\dots,p(2d_n+1)$. Set up a system of
linear equations to determine the coefficients of the reduced
polynomial $R(x) = a_0 + a_1x + \dots + a_dx^d$ that is functionally
equivalent to $p^{-1}$, where $d=d_n$. The system has the form
\[
  \begin{bmatrix}
     1 & p(1) & p(1)^2 & \dots   & p(1)^d \\
     1 & p(3) & p(3)^2 & \dots   & p(3)^d \\
     \vdots & \vdots & \vdots & \ddots & \vdots\\
     1 & p(2d+1) & p(2d+1)^2 & \dots   & p(2d+1)^d \\
  \end{bmatrix}
  \begin{bmatrix} a_0 \\ a_1 \\ \vdots \\ a_d \end{bmatrix}
 = \begin{bmatrix} 1 \\ 3 \\ \vdots \\ 2d+1 \end{bmatrix}.
\]
We apply row reduction to this system. The crucial observation is
that since, for every $a,b \in Q_n$,
\[ P(a) - P(b) = (a-b)k_{a,b}, \]
where $k_{a,b}$ is an odd number (see the proof of
Proposition~\ref{pp_n}) and odd numbers are units in $\Z_{2^n}$
the row reduction will eventually lead to a system in which the
matrix of the system has the form~(\ref{triangular}). This system
has unique solution that can be found by back substitution.
\end{proof}

\begin{example}
Let $n=4$ and $P(x) = 5 + x + x^2$. The polynomial $P(x)$ induces a
permutation $p$ on $Q_4$. We will find the unique reduced polynomial
$R(x) = a_0 + a_1 x + a_2x^2$, with $0 \leq a_0 \leq 15$, $0 \leq
a_1 \leq 7$, and $0 \leq a_2 \leq 1$, that induces the inverse
permutation $p^{-1}$ on $Q_n$.

We calculate $p(1)=7$, $p(3)=1$ and $p(5)=3$. We then perform row
reduction (over $\Z_{16}$) on the system
\[
 \begin{bmatrix}
 1 & 7 & 1 &| & 1 \\
 1 & 1 & 1 &| & 3 \\
 1 & 3 & 9 &| & 5
 \end{bmatrix} \sim
 \begin{bmatrix}
 1 & 7  & 1 &| & 1 \\
 0 & 10 & 0 &| & 2 \\
 0 & 12 & 8 &| & 4
 \end{bmatrix} \sim
 \begin{bmatrix}
 1 & 7  & 1 &| & 1 \\
 0 & 2 & 0 &|  & 10 \\
 0 & 4 & 8 &|  & 12
 \end{bmatrix} \sim
\begin{bmatrix}
 1 & 7  & 1 &| & 1 \\
 0 & 2 & 0 &|  & 10 \\
 0 & 0 & 8 &|  & 8
 \end{bmatrix},
\]
where the third matrix is obtained from the second by re-scaling the
second row by $13=5^{-1}$ and the third row by $11=3^{-1}$. The last
system is triangular and has unique solution $a_2=1$ $a_1=5$ and
$a_0=13$. Thus $R(x) = 13 + 5x+x^2$ induces the inverse polynomial
function $p^{-1}$.
\end{example}

\begin{theorem}
There exists an algorithm of polynomial complexity in $n+m$ that,
given a polynomial $P(x) \in \P_n$ of degree $m$ (with coefficients reduced modulo $2^n$), produces the
multiplicative inverse $\frac{1}{P(x)}$ in reduced form.
\end{theorem}

\begin{proof}
To calculate the reduced polynomial $S(x) = \frac{1}{P(x)}$ it
suffices to calculate $p(x)$ for $x=1,3,\dots,2d_n+1$, then
calculate the multiplicative inverses $s(x)=\frac{1}{p(x)}$, for
$x=1,3,\dots,2d_n+1$, and finally use Theorem~\ref{interpolation} to
find the coefficients of $S(x)$.
\end{proof}

\section{Huge quasigroups defined by polynomial
functions}\label{s:quasigroups}

A $k$-groupoid $(k\ge 1)$ is an algebra $(Q,f)$ on a nonempty set
$Q$ as its universe and with one $k$-ary operation $f:Q^k\to Q$.

\begin{definition}
A $k$-groupoid $(Q,f)$ is said to be a \emph{$k$-quasigroup} if any
$k$ out of any $k+1$ elements $a_1,a_2,\dots,a_{k+1}\in Q$ satisfying the
equality
\[
 f(a_1,a_2,\dots,a_k)=a_{k+1}
\]
uniquely determine the remaining one. 

A $k$-groupoid is said to be a
\emph{cancellative} $k$-groupoid if it satisfies the cancellation
law
\[
 f(a_1,\dots,a_{i-1},x,a_{i+1},\dots,a_k)=f(a_1,\dots,a_{i-1},y,a_{i+1},\dots,a_k)\Rightarrow
x=y,
\]
for each $i=1,\dots,k$ and all
$x,y,a_1,\dots,a_{i-1},a_{i+1},\dots,a_k$ in $Q$.
\end{definition}

For $k=2$ we obtain the standard notion of a quasigroup. 

The definition of a $k$-quasigroup immediately implies the
following. Let $(Q,f)$ be a finite $k$-quasigroup and let the map
$\varphi:Q\to Q$ be defined by
$\varphi(x)=f(a_1,\dots,a_{i-1},x,a_{i+1},\dots,a_k)$, for some
fixed $a_1,\dots,a_{i-1},a_{i+1},\dots,a_k$ in $Q$. Then $\varphi$
is a permutation on $Q$.

Here we consider only finite $k$-quasigroups $(Q,f)$, i.e., $Q$ is a
finite set, and in this case we have the following property
(\cite{smile}).

\begin{proposition}
The following statements are equivalent for a finite $k$-groupoid
$(Q,f)$:

(a) \quad $(Q,f)$ is a $k$-quasigroup.

(b) \quad $(Q,f)$ is a cancellative $k$-groupoid. \label{uslov za
kvaz}
\end{proposition}

Given a $k$-quasigroup $(Q,f)$ we can define $k$ new $k$-ary
operations $f_i, \ i=1,2,\dots, k,$ by
\[
 f_i(a_1,\dots,a_k)=b\ \Longleftrightarrow \
 f(a_1,\dots,a_{i-1},b,a_{i+1},\dots,a_k)=a_i.
\]
These operations are called adjoint operations of
$f$. Then $(Q,f_i)$ are $k$-quasigroups as well (\cite{belousov}).

\begin{definition}
A \emph{huge} $k$-quasigroup is said to be a $k$-quasigroup $(Q,f)$
such that all of the operations $f,f_1,f_2,\dots,f_k$ can be
computed with complexity ${\mathcal O}(\log(|Q|)^{\alpha})$ for some
constant $\alpha$.
\end{definition}

The problem of effective constructions of quasigroups of any order
can be solved, for example, by using P. Hall's algorithm for
choosing different representatives for a family of sets. The
algorithm is of complexity ${\mathcal O}(n^3)$, where $n$ is the
order of the quasigroup, and is not applicable for, let say,
$n=2^{16}$. We will show here how the permutational polynomial
functions from $\PF_n$ can be used in order to construct families of
huge quasigroups on the sets $Q_n$ and $\Z_{2^n}$.

\begin{theorem}\label{quasi-on-Q}
Let $p_1,p_2,\dots,p_k$ be permutations in $\PPF_n$. Define a
$k$-ary operation $f$ on $Q_n$ by
\begin{equation}\label{*}
 f(a_1,a_2,\dots,a_k)=p_1(a_1)p_2(a_2)\cdots p_k(a_k) \pmod{2^n}.
\end{equation}
Then the $k$-groupoid $(Q_n,f)$ is a huge quasigroup.
\end{theorem}

\begin{proof}\ Let $r=2^n$. The permutations in $\PPF_n$ are defined by
polynomials $P(x)$ of degree smaller than $(\log_2 r+1+\lfloor
\log_2(\log_2 r)\rfloor)/2$ (by Proposition~\ref{d_n bound}). Then
the evaluation of $P(x)$ modulo $2^n$ can be computed in polynomial
complexity with respect to $\log_2 r$. Consequently, the function
$f$ defined by (\ref{*}) can be computed in polynomial complexity
with respect to $\log_2 r$.

Consider now the adjoint operations $f_i$ of $f$. We have, for any
$a_1,a_2,\dots,a_k,b\in Q_n$:
\begin{alignat*}{2}
 &&&f_i(a_1,a_2,\dots,a_k)=b\ \Longleftrightarrow\\
 &&&\quad\Longleftrightarrow\ f(a_1,\dots,a_{i-1},b,a_{i+1},\dots,a_k)=a_i\\
 &&&\quad\Longleftrightarrow\ p_1(a_1)\cdots p_{i-1}(a_{i-1})p_i(b)p_{i+1}a_{i+1}\cdots p_k(a_k)=a_i\\
 &&&\quad\Longleftrightarrow\ p_i(b)=(p_{i-1}(a_{i-1}))^{-1}\cdots (p_{1}(a_{1}))^{-1}a_i(p_{k}a_{k})^{-1}\cdots (p_{i+1}(a_{i+1}))^{-1}\\
 &&&\quad\Longleftrightarrow\ b=p_i^{-1}((p_{i-1}(a_{i-1}))^{-1}\cdots (p_{1}(a_{1}))^{-1}a_i(p_{k}a_{k})^{-1}\cdots (p_{i+1}(a_{i+1}))^{-1})
\end{alignat*}
By using the Hensel lifting technique the inverse elements
$(p_{j}(a_{j}))^{-1}$ can be computed in polynomial complexity with
respect to $\log_2 r$ (see Section\ref{s:group}), and the same is
true for the inverse permutation $p_i^{-1}$ by Theorem~\ref{inverse
perm}.
\end{proof}

\begin{theorem}
Let $p_1,p_2,\dots,p_k$ be permutations in $\PPF_n$. Define a
$k$-ary operation $f$ on $\Z_{2^n}$ by
\begin{equation}
 f(a_1,a_2,\dots,a_k)=\hat{p_1}(a_1)+\hat{p_2}(a_2)+\cdots+\hat{ p_k}(a_k) \pmod{2^n}.
\end{equation}
where $\hat{p_i}$ are defined by (\ref{p kapa}). Then the
$k$-groupoid $(Q_n,f)$ is a huge quasigroup.
\end{theorem}
\begin{proof}
The proof is similar to the proof of Theorem~\ref{quasi-on-Q}. We
only need to note that the inverse permutation
\[
 \hat{p_i}^{-1}=
 \begin{cases}
   p_i^{-1}(a), &a\in Q_n\\
   p_i^{-1}(a+1)-1, &a\in Q_n'
 \end{cases}
\]
can be computed in polynomially complexity with respect to $\log_2
r$.
\end{proof}

\begin{theorem}
Let $p_1,p_2,\dots,p_k$ and $h_1,h_2,\dots,h_k$ be permutations in
$\PPF_n$. Define a $k$-ary operation $f$ on $\Z_{2^n}$ by
\begin{equation}
 f(a_1,a_2,\dots,a_k)=f_{p_1,h_1}(a_1)+f_{p_2,h_2}(a_2)+\cdots+f_{p_k,h_k}(a_k) \pmod{2^n}.
\end{equation}
where $f_{p_i,h_i}$ are defined by (\ref{p h kapa}). Then the
$k$-groupoid $(Q_n,f)$ is a huge quasigroup.
\end{theorem}

We note that Rivest~\cite{rivest} gives a simple necessary and
sufficient condition for a bivariate polynomial $P(x,y)$ modulo
$2^n$ to represent a quasigroup on $\Z_{2^n}$, namely $P(x,0),\
P(x,1),\ P(0,y)$ and $P(1,y)$ should be univariate permutational
polynomials on $\Z_{2^n}$. This result is based on his main result
in~\cite{rivest} (see Theorem~\ref{t:rivest} in Section~\ref{s:rivest}).

\end{document}